\documentclass[12pt]{amsart}
\usepackage[latin1]{inputenc}	

\usepackage[english]{babel}
\usepackage{amsmath,amsthm}
\usepackage{amssymb,latexsym}
\usepackage{graphicx}

\newtheorem{theorem}{Theorem}[section]

\theoremstyle{remark}

\allowdisplaybreaks

\begin{document}
\title{A classical functional generalization of the first Barnes Lemma}

\author{{\sc Raffaele Marcovecchio}}
\address{Dipartimento di Ingegneria e Geologia\\
Universit\`a di Chieti-Pescara \\
Viale Pindaro 42 \\
65127 Pescara \\
Italy }

\email{raffaele.marcovecchio@unich.it}

\begin{abstract}
We give a brief account and a simpler proof of a contour integral formula 
for the Gauss hypergeometric function. Such formula is alternative to Barnes's
integral formula and generalizes the first Barnes Lemma.  
\end{abstract}

   \subjclass[2010]{Primary 33C60; Secondary 33C05}

   \keywords{Mellin-Barnes integrals, hypergeometric function.}

\maketitle
\section{Introduction}
The Gauss hypergeometric function (denoted by $F(a,b;c)$ throughout 
the present paper) has been deeply studied, and several integral 
representations can be found in books dealing with special functions 
(see e.g.~\cite[Sections 8.3, 8.8]{viola}). An important integral 
was discovered by Barnes (see formula (\ref{barnes}) below), 
who build an alternative theory of the function $F(a,b;c)$ based 
on such integral formula. One useful feature of formulas of the type 
(\ref{barnes}) relies in the possibility of applying the saddle point 
method to obtain a precise asymptotic estimate of the function involved 
(see the monography~\cite{pariskaminski}). Another interesting property 
of (\ref{barnes}) is that it possesses a wide range of extensions 
to generalized hypergeometric series (see~\cite[Sections 4.6, 4.7]{slater}).

The countour integral formula proved in the present paper is not new 
(see~\cite[Section 14.53]{whittakerwatson} and~\cite[formula (15.6.7)]{nist}).
However, we believe it is worth the present short note, because our proof 
appears to be simpler than that in~\cite{whittakerwatson}, 
and is independent of Barnes's integral formula (\ref{barnes}).
We remark that formula (\ref{maintheorem}) encompasses (and,
in the present note, relies on) the first Barnes lemma 
(see (\ref{firstbarnes}) below), whose proof in~\cite{bailey} 
is very similar to the proof of Barnes's integral 
formula (\ref{barnes}). Therefore our contribution allows one to use 
the residue theorem in the proof of the first Barnes Lemma only. After 
that, one can prove the contour integral formula (\ref{maintheorem}) 
as in the present paper, and finally combine the two results to prove 
(\ref{barnes}), with an argument similar 
to~\cite[Section 14.53]{whittakerwatson}, without applying the residue 
theorem a second time, as in~\cite{whittakerwatson} . Also, our argument 
is very simple but apparently has been generally overlooked in this context, 
and may have further applications.

The first Barnes lemma is often considered 
as an integral analogue of the Gauss summation formula
\begin{equation}			\label{gauss}
F(a,b;c;1)= \frac{\Gamma(c) \Gamma(c-a-b)}{\Gamma(c-a) \Gamma(c-b)}.
\end{equation}
In addition, formula (\ref{maintheorem}) can be seen as an integral 
analogue of the formula connecting the values of hypergeometric 
functions of $z$ and $1-z$ (see (\ref{zetatoneminus}) below), 
and this is precisely the context where (\ref{maintheorem}) 
is used in~\cite[Section 14.53]{whittakerwatson}. 
Let us also point out two formulas close to (\ref{maintheorem}):
the first one, obtained in 1939 by S.O. Rice for 
his function $H_n(\xi,p;v)={}_3 F_2 (-n,n+1,\xi;1,p;v)$ 
(see~\cite[Vol I, p.193]{bateman}), and the second one, usually used 
in the proof af the second Barnes lemma (see e.g.~\cite[p.43]{bailey}. 
We mention these formulas at the end of the present paper.
\section{The main result and  a few similar formulas}
We denote by $(\xi)_n$ the product $\xi(\xi+1)\cdots (\xi+n-1)$ 
for any complex number $\xi$ and for any $n=1,2,\dots$, and 
we put $(\xi)_0=1$. We say that $\xi$ is {\it admissible} if 
$\xi$ is not a negative integer nor $0$.

The Gauss hypergeometric function $F(a,b;c;z)$ is defined over the 
unit disc $|z|<1$ in the complex plane by the series 
\begin{equation}			\label{hyperdefinit}
F(a,b;c;z) = \sum_{k=0}^\infty \frac{(a)_k (b)_k }{k! (c)_k } z^k,
\end{equation}
where $a$, $b$ and $c$ are complex numbers and $c$ is admissible. 
Note that the series $F(a,b;c;z)$ may terminate: this happens 
when $a$ or $b$ are not admissible. In this case the function 
(\ref{hyperdefinit}) is a polynomial in $z$, and could be defined 
even if $c$ is not admissible, provided that $\min\{a,b\}\le c$.

Let $\Gamma(z)$ be the Euler gamma function, defined in the complex 
half-plane ${\rm Re }\, z>0$ by 
\[
\Gamma(z)=\int\limits_0^\infty t^{z-1} e^{-t} {\rm d} t,
\]
and extended to a meromorphic function in the complex plane, with 
simple poles at $z=-n$ with residue $\frac{(-1)^n}{n}$ ($n=0,1,2,\dots$), 
for example by splitting the integration path $(0,\infty)$ in the union 
of $(0,1)$ and $(1,\infty)$. Two main properties of the function $\Gamma(z)$ 
are important in the following: the Stirling formula 
\[
\log \Gamma(z) = (s-\frac{1}{2}) \log s -s +\frac{1}{2} \log (2\pi) + o(1), 
\]
valid for $|\arg z|<\pi-\delta $ for any $\delta>0$, and 
the functional equations
\[
\Gamma(z) \Gamma(1-z)= \frac{\pi z}{\sin \pi z}, \qquad \Gamma(z+1)=z\Gamma(z).
\]

The Barnes integral representation 
(see e.g.~\cite[Theorem 2.4.1]{andrewsaskeyroy}) 
of the function (\ref{hyperdefinit}) is given by
\begin{equation}		\label{barnes}
\frac{\Gamma(a) \Gamma(b)}{\Gamma(c)} F(a,b;c;z) = 
\frac{1}{2\pi i} \int\limits_{-i \infty}^{i \infty} 
		\frac{\Gamma(a+s) \Gamma(b+s) \Gamma(-s)}{\Gamma(c+s)} (-z)^s {\rm d} s,
\end{equation}
valid under the conditions that $|z|<1$, $z\not=0$ and $|\arg(-z)|<\pi$, 
and that $a$, $b$ and $c$ are admissible. The path $L$ of integration is 
curved, if necessary, in such a way that separates the poles $s=-a-n$ and 
$s=-b-n$ ($n=0,1,2,\dots$) at the left of $L$ from the poles $s=0,1,2,\dots$
at the right of $L$. In the sequel, we denote by $F(a,b;c;z)$ the analytic
function defined for $z\notin [1,\infty)$ either by the series (\ref{hyperdefinit}), if $|z|<1$, or by the integral (\ref{barnes}), 
if $z\notin [0,\infty)$. 

The first Barnes lemma (see e.g.~\cite[Theorem 2.4.2]{andrewsaskeyroy}) 
states that 
\begin{equation}		\label{firstbarnes}
\frac{1}{2\pi i} \int\limits_{-i \infty}^{i \infty} 
\Gamma(a+s) \Gamma(b+s) \Gamma(c-s) \Gamma(d-s) = 
\frac{\Gamma(a+c) \Gamma(a+d) \Gamma(b+c) \Gamma(b+d)}{\Gamma(a+b+c+d)},
\end{equation}
provided that $a+c$, $a+d$, $b+c$ and $b+d$ are admissible.

Using (\ref{barnes}) and (\ref{firstbarnes}) one can prove 
(see~\cite[Sect. 14.53]{whittakerwatson}) that
\begin{multline}		\label{zetatoneminus}
F(a,b;c;z) = 
\frac{\Gamma(c) \Gamma(c-a-b)}{\Gamma(c-a) \Gamma(c-b)} F(a,b;1+a+b-c;1-z) \\ +
\frac{\Gamma(c) \Gamma(a+b-c)}{\Gamma(a) \Gamma(b)} (1-z)^{c-a-b} 
													F(c-a,c-b;1+c-a-b;1-z)
\end{multline}
Using (\ref{firstbarnes}) we can prove an integral formula that 
encompasses (\ref{zetatoneminus}), which is a generalization of 
(\ref{gauss}). For this reason we named formula (\ref{maintheorem}) 
below a functional generalization of the first Barnes lemma.
\begin{theorem}\cite[Section 14.53]{whittakerwatson}
Let $a$, $b$, $c$ and $z$ be complex numbers such that $z\notin (-\infty,0]$, 
and that $a$, $c-a$, $b$, $c-b$ and $c$ are admissible. Then
\begin{equation}		\label{maintheorem}
\frac{\Gamma(a) \Gamma(c-a) \Gamma(b) \Gamma(c-b)}{\Gamma(c)} F(a,b;c;1-z)
= \frac{1}{2\pi i} \int\limits_{-i \infty}^{i \infty} 
	\Gamma(a+s) \Gamma(b+s) \Gamma(c-a-b-s) \Gamma(-s) z^s {\rm d} s,
\end{equation} 
where the integration path $L$ separates the poles $s=-a-n$ and $s=-b-n$ 
$(n=0,1,2,\dots)$ on the left of $L$ from the poles $s=n$ and $s=a+b-c+n$ 
$(n=0,1,2,\dots)$ on the right of $L$.
\end{theorem}
\begin{proof}
Suppose that $|1-z|<1$. For any $n=0,1,2,\dots$ we have
\[
(-1)^n \frac{{\rm d}^n}{{\rm d} z^n} F(a,b;c;1-z)  \Big|_{z=1} 
= \frac{(a)_n (b)_n}{(c)_n} 
= \frac{\Gamma(a+n)}{\Gamma(a)} \frac{\Gamma(b+n)}{\Gamma(b)} 
										\frac{\Gamma(c)}{\Gamma(c+n)}.
\]
The integral at the right-hand side of (\ref{maintheorem}) is an analytic 
function in the domain $|\arg z|<2 \pi$ (see~\cite[Lemma 2.4]{pariskaminski}),
which plainly contains the disc $|1-z|<1$. This implies that the derivative 
of the integral in (\ref{maintheorem}) with respect to $z$ equals 
\[
\frac{1}{2\pi i} \int\limits_{-i \infty}^{i \infty} 
	\Gamma(a+s) \Gamma(b+s) \Gamma(c-a-b-s) \Gamma(-s) s z^{s-1} {\rm d} s,
\]
this being an integral of the same type as in (\ref{maintheorem}), 
once it is noticed that $-s \Gamma(-s) = \Gamma(1-s)$, and after 
substituting the variable $s$ with $t$ by putting $s=1+t$, and then 
renaming $t$ with $s$. We thus have
\begin{multline*}
(-1)^n \frac{{\rm d}^n}{{\rm d} z^n}
				\frac{1}{2\pi i} \int\limits_{-i \infty}^{i \infty} 
				\Gamma(a+s) \Gamma(b+s) \Gamma(c-a-b-s) \Gamma(-s) z^s {\rm d} s
				\Big|_{z=1}		\\
= \frac{1}{2\pi i} \int\limits_{-i \infty}^{i \infty} 
				\Gamma(a+s) \Gamma(b+s) \Gamma(c-a-b-s) \Gamma(n-s) {\rm d} s
				\qquad (n=0,1,2,\dots).
\end{multline*}
By (\ref{firstbarnes}) the last integral equals
\[
\frac{\Gamma(c-a) \Gamma(c-b) \Gamma(a+n) \Gamma(b+n)}{\Gamma(c+n)},
\]
therefore (\ref{maintheorem}) is proved for $|1-z|<1$, because all 
the derivatives of both sides of (\ref{maintheorem}) coincide at $z=1$. 
By analytic continuation (\ref{maintheorem}) holds for $z\notin(-\infty,0]$.
\end{proof}
From (\ref{maintheorem}), using Stirling's formula, the residue theorem, and
changing $z$ into $1-z$, after a few simplifications one easily gets
(\ref{zetatoneminus}), very much as in the standard proofs of (\ref{barnes}) 
and (\ref{firstbarnes}). Of course, it is possible to go the other way, 
which is the usual proof of (\ref{maintheorem}).

Let us finish this short paper with two formulas formally close to
(\ref{maintheorem}): the first one (see~\cite[p.43]{bailey}) is
\begin{multline*}
\sum_{n=0}^\infty 
	\frac{(\alpha_1)_n (\alpha_2)_n (\alpha_3)_n }{n! (\beta_1)_n (\beta_2)_n} 
= \frac{\Gamma(\beta_1)}{\Gamma(\alpha_1) \Gamma(\beta_1-\alpha_1)
						\Gamma(\alpha_2) \Gamma(\beta_1-\alpha_2)} \\ \times
\frac{1}{2\pi i} \int\limits_{-i \infty}^{i \infty} 
	\Gamma(\alpha_1+s) \Gamma(\alpha_2+s) 
				\Gamma(\beta_1-\alpha_1-\alpha_2-s) \Gamma(-s) 
				F(\alpha_3,-s;\beta_2;1){\rm d} s,
\end{multline*}
and is used in the standard proof of the second Barnes lemma. As to the 
second one, let us consider (see~\cite[Vol I, p.193]{bateman}) the sequence 
of polynomials
\[
H_n(\xi,p;v)= \sum_{j=0}^n \frac{(-n)_j (n+1)_j (\xi)_j}{j!^2 (p)_j}\, v^j 
\qquad (n=0,1,2,\dots).
\]
Here $\xi$, $p$ and $v$ are complex numbers and $p+n+1$ is admissible. 
Then
\begin{multline*}
\Gamma(p-q) \Gamma(q) \Gamma(p-\xi) \Gamma(\xi) H_n(\xi,p;v) 	\\
= \frac{\Gamma(p)}{2\pi i} 
	\int\limits_{\sigma-i \infty}^{\sigma+i \infty} 
	\Gamma(s) \Gamma(q-s) \Gamma(\xi-s) \Gamma(p-q-\xi+s) H_n (s,q;v) {\rm d} s, 
\end{multline*}
where $0<{\rm Re}\, \sigma<{\rm Re}\, q$ 
and $0<{\rm Re} (\xi - \sigma) < {\rm Re} (q-p)$. It is worth noticing that 
the generating function of the sequence $H_n$ is
\[
\sum_{n=0}^\infty t^n H_n(\xi,p;v)= \frac{1}{1-t} F(\xi,1/2;p; -4vt (1-t)^{-2}).
\]

\end{document}